\documentclass{article}
  
\usepackage{amsfonts}
\usepackage{amsmath}
\usepackage{amssymb}
\usepackage{wasysym}
\usepackage{yfonts}  
\usepackage{amsthm}
\usepackage{amscd}
\usepackage[all]{xypic}
\usepackage{mathrsfs}

\CompileMatrices

\newtheorem{theorem}{Theorem}[section]
\newtheorem{lemma}[theorem]{Lemma}
\newtheorem{proposition}[theorem]{Proposition}
\newtheorem{corollary}[theorem]{Corollary}
\newtheorem*{definition}{Definition}
\newtheorem*{example}{Example}
\newtheorem*{remark}{Remark}
\newtheorem*{notation}{Notation}
\newtheorem*{acknowledgement}{Acknowledgements}
\newtheorem*{utheorem}{Theorem}
\numberwithin{equation}{section}

\begin{document}
\title{Groups with the same cohomology as their profinite completions}
\author{
{\sc Karl Lorensen}\\
\\
Mathematics Department\\
Pennsylvania State University, Altoona College\\
3000 Ivyside Park\\
Altoona, PA 16601-3760\\
USA\\
e-mail: {\tt kql3@psu.edu}
}

\maketitle

\begin{abstract} For any positive integer $n$, $\mathcal{A}_n$ is the class of all groups $G$ such that, for $0\leq i\leq n$, $H^i(\hat{G},A)\cong H^i(G,A)$ for every finite discrete $\hat{G}$-module $A$. We describe  certain types of free products with amalgam and HNN extensions that are in some of the classes $\mathcal{A}_n$. In addition, we investigate the residually finite groups in the class $\mathcal{A}_2$.

\vspace{12pt}

\end{abstract}

\section{Introduction}

\indent

If $G$ is a group, let $\hat{G}$ denote the profinite completion of $G$ and $c_G:G\to \hat{G}$ the completion map. In this paper we study groups $G$ such that, for any finite discrete $\hat{G}$-module $A$, the map induced by $c_G$ from the continuous cohomology group $H^i(\hat{G},A)$ to the discrete cohomology group $H^i(G,A)$ is an isomorphism for a range of dimensions $i$. When this occurs for $0\leq i\leq n$, we identify the group as lying in the class $\mathcal{A}_n$.  This definition makes both $\mathcal{A}_0$ and $\mathcal{A}_1$ equal to the class of all groups. However, for $n\geq 2$ membership in $\mathcal{A}_n$ becomes more restricted:  $\mathbb Q/\mathbb Z$ and ${\rm PGL}_2(\mathbb C)$, for example, are both not in $\mathcal{A}_2$. Nevertheless, a number of important types of groups are known to belong to all the classes $\mathcal{A}_n$; these include polycyclic groups, free groups and, of course, finite groups. 

The classes $\mathcal{A}_n$ have received intermittent attention in the literature since at least the early 1960s.  The first mention of these groups appears to have been by J-P. Serre [{\bf 13}, Exercises 1, 2,  Chapter 2], who established most of their basic properties. Later, H. Schneebeli \cite{schneebeli} examined the class $\mathcal{A}_2$ by employing the virtual cohomology of a group. Moreover, more recently, in his work on Grothendieck's question on Brauer groups, S. Schr\"oer \cite{schroeer} invoked groups that are in $\mathcal{A}_n$ for all $n\geq 0$.

From the perspective of a group theorist the class $\mathcal{A}_2$ is particularly interesting.  As observed in \cite{serre}, this class consists precisely of those  groups $G$ such that, for every group extension 
$$N\stackrel{\iota}{\rightarrowtail} E\stackrel{\epsilon}{\twoheadrightarrow}  G$$ 
with $N$ finitely generated, the sequence of profinite completions  
$$1\to \hat{N} \stackrel{\hat{\iota}}{\to} \hat{E}\stackrel{\hat{\epsilon}}{\to} \hat{G}\to 1$$ 
is exact. This is equivalent to the assertion that the profinite topology on $E$ induces the full profinite topology on  $N$; in other words, for every $H\leq N$ with $[N:H]<\infty$, there exists $K\leq E$ with $[E:K]<\infty$ such that $K\cap N\leq H$.

As shown in Section 2, the above property of groups in $\mathcal{A}_2$ reveals that the residually finite groups in that class are of a remarkable species:  a group $G$ is both residually finite and in $\mathcal{A}_2$ if and only if every extension of a finitely generated residually finite group by $G$ is again residually finite. Such groups are the the subject of J. Corson and T. Ratkovitch's article \cite{corson}, where they are labelled  ``super residually finite." Modifying Corson and Ratkovitch's term slightly, we will refer to groups that are both residually finite and in $\mathcal{A}_2$ as {\it highly residually finite}. The highly residually finite groups include all polycyclic groups and free groups; moreover, the class of finitely generated highly residually finite groups is closed under the formation of extensions.

Although the classes $\mathcal{A}_n$ for $n\geq 2$ are endowed with appealing properties, it is in general very difficult to determine whether a given group belongs to $\mathcal{A}_n$. 
The goal of the present article is to remedy this situation by identifying some new types of groups in the classes $\mathcal{A}_n$, particularly ones that are highly residually finite. Our main focus is on free products with amalgamation and HNN extensions formed from groups already known to belong to $\mathcal{A}_n$. In Section 3 we employ the Mayer-Vietoris sequences to ascertain conditions under which such constructions yield groups that are again in $\mathcal{A}_n$. For instance, one consequence of our results is that any free product of two virtually polycyclic groups with cyclic amalgam is in $\mathcal{A}_n$ for all $n\geq 0$.  In addition, we deduce that any right-angled Artin group is in $\mathcal{A}_n$ for all $n\geq 0$.

The last section, Section 4, is devoted to proving that certain ascending HNN extensions are highly residually finite. Ascending HNN extensions are HNN extensions that are formed from a monomorphism of the base group. More precisely, if $\phi:G\to G$ is a monomorphism, then the ascending HNN extension of $G$ with respect to $\phi$, denoted $G_{\phi}$, is given by
$$G_{\phi} = \langle G, t \ |\  t^{-1}gt = \phi(g)\ \mbox{for all}\ g\in G\rangle.$$
These constructions have already been studied extensively with regard to their residual finiteness in \cite{sapir, rhem, wise}.  For instance, it is known that having a base group that is either polycyclic or free of finite rank ensures that an ascending HNN extension must be residually finite. Generalizing these results, we prove that any ascending HNN extension with a base group of one of these two types is, in fact, highly residually finite. In the process, we furnish a new proof, significantly shorter than those already present in the literature, that an ascending HNN extension of a polycyclic group is residually finite.  Finally, in the following theorem, we establish our most general result concerning the highly residual finiteness of ascending HNN extensions.

\begin{utheorem} Let
$$1=G_1\unlhd G_2\unlhd G_3\unlhd \cdots G_{r-1}\unlhd G_r=G$$
be a normal series in a group $G$ such that each factor group $G_i/G_{i-1}$ is either finite, free of finite rank or polycyclic. If $\phi:G\to G$ is a monomorphism such that $\phi(G_i)\leq G_i$ for $1\leq i\leq r$, then $G_{\phi}$ is highly residually finite.
\end{utheorem}

\vspace{20pt}

\begin{notation}{\rm  We write $H\leq_fG$  (respectively $H\unlhd_f N$) when $H$ is a subgroup (respectively normal subgroup) of finite index in the group $G$.

Since we will be working with both continuous and discrete cohomology, it will be advantageous to distinguish between the two types in our notation. Henceforth we will employ $H^*(\  \ ,\  \ )$ for discrete cohomology and $H^*_c(\  \ ,\  \ )$ for continuous cohomology.}

\end{notation}

\vspace{20pt}

\section{Elementary properties of the classes $\mathcal{A}_n$} 

\vspace{20pt}

In this section we discuss the elementary properties of the classes $\mathcal{A}_n$. We begin with the following lemma, due to Serre  [{\bf 13}, Exercise 1a, Chapter 2].

\begin{lemma}{\rm (Serre)} Assume $n$ is a nonnegative integer. For a group $G$ the following statements are equivalent.
\vspace{5pt}

(i) For any finite discrete $\hat{G}$-module $A$, the map $H_c^i(\hat{G},A)\to H^i(G,A)$ induced by $c_G$ is bijective for $0\leq i\leq n$ and injective for $i=n+1$.
\vspace{5pt}

(ii) For any finite discrete $\hat{G}$-module $A$, the map $H_c^i(\hat{G},A)\to H^i(G,A)$ induced by $c_G$ is surjective for $0\leq i\leq n$.
\vspace{5pt}
\end{lemma}

\begin{remark}{\rm For the convenience of  those readers who wish to work through the details of Serre's proof in [{\bf 13}, Exercise 1a, Chapter 2], we point out that there is an error in his definition of the group-theoretic property he identifies as $\mbox{C}_n$: the module $M'$ should be in the class $C_K'$; that is, it must be finite. This misprint also appears in the French editions of \cite{serre}.}
\end{remark}

Serre's lemma forms the basis for his definition of the classes $\mathcal{A}_n$.

\begin{definition} {\rm Let $n$ be a nonnegative integer. The class $\mathcal{A}_n$ is the class of all groups $G$ satisfying the two equivalent properties in Lemma 2.1.}
\end{definition}

We can at once make the following assertion about the classes $\mathcal{A}_0$ and $\mathcal{A}_1$.

\begin{proposition}{\rm (Serre [{\bf 13}, Exercise 1b, Chapter 2])} The classes $\mathcal{A}_0$ and $\mathcal{A}_1$ are each equal to the class of all groups.
\end{proposition}

In addition, we can immediately cite some examples of groups that lie in all of the classes $\mathcal{A}_n$.

\begin{proposition} Finite groups and free groups are in the class $\mathcal{A}_n$ for all $n\geq 0$.
\end{proposition}

\begin{proof} That finite groups are in $\mathcal{A}_n$ is obvious. Moreover, if $G$ is free, then, for any finite discrete $\hat{G}$-module $A$, both $H^i(G,A)$ and $H_c^i(\hat{G},A)$ are trivial for $i\geq 2$. Thus every free group is in $\mathcal{A}_n$.  
\end{proof}

Serre \cite{serre} observes that the class $\mathcal{A}_2$ has an intriguing property.
In order to appreciate this property, we will require the following notion.

\begin{definition}{\rm  If $H\leq G$, we say that $H$ is {\it topologically embedded} in $G$ if the subspace topology induced on $H$ by the profinite topology on $G$ coincides with the full profinite topology on $H$. 

If $H$ is a topologically embedded subgroup in $G$, we write $H\leq_t G$; if, in addition, $H$ is normal, we write $H\unlhd_t G$. }
\end{definition}

Note that the following three assertions concerning a subgroup $H$ of a group $G$ are equivalent:

(i) $H\leq_t G$;

(ii) for each $K\leq_f H$, there exists $L\leq_f G$ such that $L\cap H\leq K$;

(iii) $\hat{H}$ embeds in $\hat{G}$.
\vspace{5pt}

As we shall see below, group extensions with quotients in the class $\mathcal{A}_2$ behave particularly nicely with respect to profinite completion. In general, profinite completion is a right exact functor; in other words, for any extension $$N\stackrel{\iota}{\rightarrowtail}
G\stackrel{\epsilon}{\twoheadrightarrow} Q,$$ the sequence $$\hat{N}\stackrel{\hat{\iota}}{\to} \hat{G}\stackrel{\hat{\epsilon}}{\to}  \hat{Q}\to 1$$ is exact. However, if $Q$ belongs to $\mathcal{A}_2$ and $N$ is finitely generated, then 
$$1\to \hat{N}\stackrel{\hat{\iota}}{\to} \hat{G}\stackrel{\hat{\epsilon}}{\to}  \hat{Q}\to 1$$ 
is exact; that is, $\mbox{Im}\ \iota \unlhd_t G$. This is the content of the following proposition, part of whose proof is sketched in [{\bf 13}, Exercise 2, Chapter 2].

\begin{proposition} The following statements are equivalent for a group $G$.
\vspace{5pt}

\noindent (i) The group $G$ belongs to $\mathcal{A}_2$. 
\vspace{5pt}

\noindent (ii) For every group extension
$1\to N\stackrel{\iota}{\to} E\stackrel{\epsilon}{\to}  G\to 1$ with $N$ finitely generated, the map $\hat{\iota}:\hat{N}\to \hat{E}$ is an injection.
\end{proposition}

\begin{proof} The direction ((i)$\Longrightarrow$(ii)) is proved in \cite{serre}.
\vspace{5pt}

((ii)$\Longrightarrow$(i)) We prove that the map $H_c^2(\hat{G},A)\to H^2(G,A)$ is surjective for any finite discrete $\hat{G}$-module $A$. To establish this, let $\xi \in H^2(G,A)$ and take $A\rightarrowtail G\twoheadrightarrow Q$ to be a group extension corresponding to the cohomology class $\xi$. 
Forming the profinite completions, we obtain an extension of profinite groups $A\rightarrowtail \hat{G}\twoheadrightarrow \hat{Q}$. The cohomology class in $H^2_c(\hat{G},A)$ of this extension, then, is mapped to $\xi$ by the map  $H_c^2(\hat{G},A)\to H^2(G,A)$.
\end{proof}

The above proposition enables us to establish that each class $\mathcal{A}_n$ is closed under the formation of certain extensions.

\begin{theorem} Let $n$ be a positive integer and $N\rightarrowtail G\twoheadrightarrow Q$ a group extension. Assume $N$ and $Q$ are both in $\mathcal{A}_n$ and $N$ is of type $\mbox{FP}_{n-1}$.  Then $G$ is in $\mathcal{A}_n$.
\end{theorem}

Recall that a group $G$ is of type ${\rm FP}_n$ if $\mathbb Z$ admits a projective resolution as a trivial $G$-module that is finitely generated in the first $n$ dimensions. Finite groups, polycyclic groups and finitely generated free groups are all of type ${\rm FP}_n$ for all $n\geq 0$. Moreover, for each nonnegative integer $n$, the class of groups of type ${\rm FP}_n$ is closed under the formation of extensions.

The key ingredient in the proof of Theorem 2.5 is the following lemma on spectral sequences.

\begin{lemma} Assume $\{E_r^{pq}\}$ and $\{\bar{E}_r^{pq}\}$ are first
quadrant cohomology spectral sequences, and let $\{\phi_r^{pq}\}$ be a map 
from $\{E_r^{pq}\}$ to $\{\bar{E}_r^{pq}\}$. Suppose, further, that
$\phi_2^{pq}: E_2^{pq}\to \bar{E}_2^{pq}$ is bijective for
$0\leq  p+q\leq n$ and injective for $p+q=n+1$. Then the map
$\phi_{\infty}^{pq}: E_{\infty}^{pq}\to
\bar{E}_{\infty}^{pq}$ is bijective for $0\leq p+q\leq n$ and
injective for
$p+q=n+1$.
\end{lemma}

\begin{proof} We prove by induction on $r$ that, for $r\geq 2$, the map
$\phi_r^{pq}: E_r^{pq}\to \bar{E}_r^{pq}$ is bijective for
$0\leq p+q\leq n$ and
injective for $p+q=n+1$. Assume $r>2$.  Consider the commutative diagram
$$\begin{CD}
E^{p-r+1, q+r-2}_{r-1}@>>> E^{pq}_{r-1}@>>> E^{p+r-1, q-r+2}_{r-1}\\
@VVV  @VVV @VVV\\
\bar{E}^{p-r+1, q+r-2}_{r-1}@>>> \bar{E}^{pq}_{r-1}@>>> \bar{E}^{p+r-1,
q-r+2}_{r-1}.\\
\end{CD}$$
If $0\leq p+q\leq n$, then the first and second vertical maps in this diagram are bijective, 
whereas the third is injective. This means
that, in this case, the map $E^r_{pq}\to \bar{E}^r_{pq}$ is bijective.
Now consider the case when $p+q=n+1$. Here we have that the second
vertical map is injective and the first bijective. Therefore, the
map $E_r^{pq}\to
\bar{E}_r^{pq}$ is injective.
\end{proof}

In addition, we require the following elementary fact concerning one-dimensional cohomology, whose proof we leave to the reader.

\begin{lemma} If $G$ is a group and $A$ a topological $\hat{G}$-module, then the map $H_c^1(\hat{G},A)\to H^1(G,A)$ induced by $c_G:G\to \hat{G}$ is monic.
\end{lemma}

Armed with these two lemmas, we proceed with the proof of Theorem 2.5.

\begin{proof} Assume $n\geq 2$. By Proposition 2.4, the sequence $\hat{N}\rightarrowtail \hat{G}\twoheadrightarrow \hat{Q}$ is a profinite group extension. We will employ the
Lyndon-Hochschild-Serre cohomology spectral sequences in ordinary and continuous cohomology, respectively, associated with the
extensions $N\rightarrowtail G\twoheadrightarrow Q$ and $\hat{N}\rightarrowtail \hat{G}\twoheadrightarrow \hat{Q}$.
In view of Lemma 2.6, the conclusion will follow if it holds 
 that the map
 $H_c^p(\hat{Q},H_c^q(\hat{N},A))\to H^p(Q ,H^q(N,A))$
induced by $c_G$ is bijective for $0\leq p+q\leq n$ and
injective for $p+q=n+1$. First we observe that, since $N$ is of type $\mbox{FP}_{n-1}$,  $H^q(N,A)$ is finite for $0\leq q\leq {n-1}$. Consequently, since $Q$ and $N$ are both in $\mathcal{A}_n$, we can immediately conclude that the map $H_c^p(\hat{Q},H_c^q(\hat{N},A))\to H^p(Q ,H^q(N,A))$ has the desired properties provided $q\neq n, n+1$. In the last two cases, however, $H^q(N,A)$ may not be finite, so they need to be examined separately. The homomorphism $H_c^n(\hat{N},A)^{\hat{Q}}\to H^n(N,A)^Q$ is readily seen to be bijective since the map $H_c^n(\hat{N},A)\to H^n(N,A)$ is an isomorphism. Similarly, that the map $H_c^{n+1}(\hat{N},A)\to H^{n+1}(N,A)$ is an injection ensures that the homomorphism $H_c^{n+1}(\hat{N},A)^{\hat{Q}}\to H^{n+1}(N,A)^Q$ is also monic.  All that remains, then, is to verify that the map $H_c^1(\hat{Q},H_c^n(\hat{N},A))\to H^1(Q,H^n(N,A))$ is injective; however, this follows immediately from Lemma 2.7. Therefore, the maps  $H_c^p(\hat{Q},H_c^q(\hat{N},A))\to H^p(Q ,H^q(N,A))$ all have the required properties for $0\leq p+q\leq n+1$.

\end{proof}

The case of Theorem 2.5 when $n=2$ warrants special mention.

\begin{corollary}{\rm (Schneebeli [{\bf 11}, Theorem 1 ])} If $N\rightarrowtail G\twoheadrightarrow Q$ is a group extension in which $Q$ is in $\mathcal{A}_2$ and $N$ is a finitely generated group in $\mathcal{A}_2$, then $G$ belongs to $\mathcal{A}_2$.
\end{corollary}

Theorem 2.5 also permits  the following observation.

\begin{corollary} Any group that is virtually poly-(finitely generated free) is in $\mathcal{A}_n$ for all $n\geq 0$.
\end{corollary}

In particular, finitely generated virtually free groups and virtually polycyclic groups lie in $\mathcal{A}_n$ for all $n\geq 0$. Not all finitely generated solvable groups, however, lie in $\mathcal{A}_2$, as the following example illustrates.

\begin{example}{\rm Let $U$ be the group of all upper triangular 3x3 matrices over the dyadic rationals with diagonal  $(1, 2^k,1)$, where $k\in \mathbb Z$. For $1\leq i, j\leq 3$, let $E_{ij}$ be the 3x3 matrix with a $1$ in the $(i,j)$ position and zeros everywhere else. The group $U$ is a solvable group generated by three elements: the diagonal matrix with diagonal $(1,2,1)$, the matrix $1+E_{12}$ and the matrix $1+E_{23}$.  Let $\phi:U\to U$ be the automorphism
$$ \begin{pmatrix}1 & a & b\\
0 & 2^k & c \\
0 & 0 & 1 \end{pmatrix}\mapsto \begin{pmatrix}1 & a & 2b \\
0 & 2^k & 2c \\
0 & 0 & 1 \end{pmatrix}.$$ 
Let $A$ be the matrix $1+E_{13}$ in $U$. Then $\phi$ induces an isomorphism $U/\langle A\rangle\to U/\langle A^2\rangle$. Hence, setting $G= U/\langle A^2\rangle$, we have a group extension
$$\mathbb Z/2\rightarrowtail G\twoheadrightarrow G.$$
If $G$ were in $\mathcal{A}_2$, this would yield a profinite group extension
$$\mathbb Z/2\rightarrowtail \hat{G}\twoheadrightarrow \hat{G}.$$
However, this is an impossibility since every continuous surjective endomorphism of a topologically finitely generated profinite group is an automorphism (see [{\bf 14}, Proposition 4.2.2]). Therefore, $G$ is not in $\mathcal{A}_2$.

The group $G$ was, incidentally, the first example of a finitely generated solvable group that fails to be hopfian, due to P. Hall \cite{hall}.}

\end{example}

The residually finite groups in $\mathcal{A}_2$ have a particularly striking property, as described in the following theorem.

\begin{theorem}
The following statements are equivalent for a group $G$.
\vspace{5pt}

\noindent (i) The group $G$ is residually finite and lies in the class $\mathcal{A}_2$.
\vspace{5pt}

\noindent (ii) For any group extension $N\rightarrowtail E\twoheadrightarrow G$ with $N$ finitely generated and residually finite, $E$ is residually finite.
\vspace{5pt}

\noindent (iii) For any group extension $F\rightarrowtail E\twoheadrightarrow G$ with $F$ finite,  $E$ is residually finite.
\end{theorem}

\begin{proof}  ((i)$\Longrightarrow$(ii))  Let $N\rightarrowtail E\twoheadrightarrow G$ be a group extension with $N$ finitely generated and residually finite.
Then we have a commutative diagram
\begin{displaymath} \begin{CD}
1 @>>> N @>>> E @>>> G @>>> 1\\
&& @VVc_NV @VVc_EV @VVc_GV &&\\
1 @>>> \hat{N} @>>> \hat{E} @>>> \hat{G}@>>> 1
\end{CD} \end{displaymath}
with exact rows. Since the maps $c_N$ and $c_G$ are injective, $c_E$ is also injective.
\vspace{5pt}

((ii)$\Longrightarrow$(iii)) is trivial.
\vspace{5pt}

((iii)$\Longrightarrow$(i)) That $G$ is residually finite is immediate. In order to show that $G$ lies in $\mathcal{A}_2$, let $N\rightarrowtail E\twoheadrightarrow G$ be a group extension such that $N$ is finitely generated. We will establish that $\hat{N}\rightarrowtail \hat{E}\twoheadrightarrow \hat{G}$ is exact by demonstrating that $N\unlhd_t G$. To show this, let $H\leq _f N$. Then, since $N$ is finitely generated,  $H$ contains a subgroup $M\unlhd G$ such that $[N:M]<\infty$. It follows, then, from (iii) that $E/M$ is residually finite. This implies that there exists $K\leq_f G$ containing $M$ such that $K/M\cap N/M$ is the trivial subgroup of $E/M$. Thus $K\cap N=M\leq H$. Therefore, $N\unlhd_t G$. 
 \end{proof}

As stated in the introduction, we will refer to groups that satisfy the three equivalent conditions of Theorem 2.10 as {\it highly residually finite}. In view of their remarkable properties, these are among the most interesting groups in the classes $\mathcal{A}_n$; hence they are accorded a considerable amount of attention in the sequel. For the moment, however, we remain content merely to observe that finite groups, polycyclic groups and free groups are all highly residually finite; moreover, as follows from the corollary below, the class of finitely generated highly residually finite groups is closed under the formation of extensions.

\begin{corollary} If $N\rightarrowtail G\twoheadrightarrow Q$ is a group extension such that $N$ and $Q$ are both highly residually finite and $N$ is finitely generated, then $G$ is also highly residually finite .
\end{corollary}

\begin{proof} By Theorem 2.10, $G$ is residually finite, and, by Corollary 2.8, $G$ is in $\mathcal{A}_2$.
\end{proof}

It needs to be emphasized that not all residually finite groups are highly residually finite. For example, $\mbox{SL}_3(\mathbb Z)$ is residually finite, but, in view of [{\bf 5}, Theorem I(ii)], it is not highly residually finite. Furthermore, there are groups in $\mathcal{A}_2$ that are not residually finite; the following example is one such group.

\begin{example}{\rm Let $G$ be G. Higman's \cite{higman} example of an infinite group with four generators and four relators without any proper subgroups of finite index. Since the deficiency of Higman's presentation is $0$ and $G_{ab}=1$, we have that $H_2G=0$. This means, by the universal coefficient theorem, that $H^2(G,A)=0$ for any finite $G$-module $A$. Therefore, since $\hat{G}=1$, $G$ lies in $\mathcal{A}_2$. }
\end{example}

\section{Mayer-Vietoris sequences and the classes $\mathcal{A}_n$}
\vspace{20pt}
 
 In this section we prove that certain free products with amalgamation and HNN extensions formed from groups in $\mathcal{A}_n$ yield groups that also lie in $\mathcal{A}_n$. The main tools that we employ 
are the Mayer-Vietoris sequences for free products with amalgamation and for HNN extensions, in both their ordinary and profinite incarnations. 

We begin our discussion by recalling some facts about profinite free products with amalgamation from \cite{profinite}. If $\Gamma_1$ and $\Gamma_2$ are profinite groups with a common closed subgroup $\Delta$, then we can always form the profinite free product of $\Gamma_1$ and $\Gamma_2$ with amalgamated subgroup $\Delta$; this is the pushout of $\Gamma_1$ and $\Gamma_2$ over $\Delta$ in the category of profinite groups. If $\Gamma_1$, $\Gamma_2$ and $\Delta$ are all embedded in this pushout, which is by no means always the case, then the latter is referred to as a {\it proper} profinite free product with amalgamation. Associated to proper profinite free products with amalgamation are Mayer-Vietoris sequences that relate the cohomologies of the various groups to one another.

 We are interested in the special situation when $\Gamma_1=\hat{G_1}$, $\Gamma_2=\hat{G}_2$ and $\Delta=\hat{H}$, where $G_1$ and $G_2$ are discrete groups with a shared topologically embedded subgroup $H$. In this case, the profinite completion of $G=G_1\ast _H G_2$ is the profinite free product of $\hat{G_1}$ and $\hat{G_2}$ with amalgamated subgroup $\hat{H}$. Moreover, this profinite free product with amalgam is proper if and only if both $G_1$ and $G_2$ are topologically embedded in $G$. In this case, we have a Mayer-Vietoris sequence for $\hat{G}$ relating the cohomologies of $\hat{G_1}$, $\hat{G_2}$ and $\hat{H}$. This sequence is described in the following theorem, which also illuminates the connection to the discrete Mayer-Vietoris sequence for $G$.
 
\begin{theorem} Let $G_1$ and $G_2$ be groups with a common subgroup $H$ that is topologically embedded in both groups, and let $G=G_1\ast_H G_2$.  Assume, further, that $G_1$ and $G_2$ are both topologically embedded in $G$. Then, for each discrete $\hat{G}$-module $A$ and positive integer $i$,  we have a commutative diagram
\begin{equation} \begin{CD}
H_c^{i-1}(\hat{H},A) @>>> H_c^i(\hat{G},A) @>>> H_c^i(\hat{G_1},A)\oplus H_c^i(\hat{G_2},A)  @>>> H_c^i(\hat{H},A)\\
@VVV @VVV @VVV @VVV \\
H^{i-1}(H,A) @>>> H^i(G,A) @>>> H^i(G_1,A)\oplus H^i(G_2,A)  @>>> H^i(H,A),
\end{CD} \end{equation}
in which the rows are exact and the vertical maps are induced by the completion maps for $G_1$, $G_2$, $H$ and $G$.
\end{theorem}

Theorem 3.1 allows us to deduce the following set of criteria for determining if a free product with amalgam is in $\mathcal{A}_n$.

\begin{corollary} Let $G_1$ and $G_2$ be groups with a common subgroup $H$ that is topologically embedded in both groups. Assume $G_1$ and $G_2$ are each topologically embedded in $G=G_1\ast_H~G_2$. If $G_1$ and $G_2$ are both in $\mathcal{A}_n$ and $H$ is in $\mathcal{A}_{n-1}$, then $G$ belongs to $\mathcal{A}_n$.
\end{corollary}

\begin{proof} Let $A$ be a finite discrete $\hat{G}$-module. For $0\leq i\leq n$ the first and third maps in diagram (3.1) are bijections, whereas the fourth is an injection. Therefore, the second map is surjective, placing $G$ in $\mathcal{A}_n$.
\end{proof}

We now proceed to ascertain when the two factors in a free product with amalgam are topologically embedded.

\begin{proposition} Let $G_1$ and $G_2$ be groups with a common subgroup $H$. Assume that, for each pair $\{N_1, N_2\}$ with $N_i\unlhd_f G_i$, there exists a pair $\{P_1, P_2\}$ such that $P_i\unlhd_f G_i$, $P_i\leq N_i$ and $P_1\cap H=P_2\cap H$. Then $G_1$ and $G_2$ are topologically embedded in $G_1\ast_H G_2$. 
\end{proposition}

\begin{proof} Let $G=G_1\ast_H G_2$. Assume $N_1\unlhd_f G_1$ and $N_2\unlhd_f G_2$. Then there exists a pair $\{P_1, P_2\}$ such that $P_i\unlhd_f G_i$, $P_i\leq N_i$ and $P_1\cap H=P_2\cap H$. Since  
$P_1\cap H=P_2\cap H$, $P_1H/P_1\cong P_2H/P_2$. We can then identify these two groups via this isomorphism and form the free product with amalgamation
$$\bar{G} = G_1/ P_1\ast_{P_1H/P_1} G_2/P_2.$$
As a free product of two finite groups with amalgamation, the group $\bar{G}$ is virtually free; moreover, there is a homomorphism $\theta: G\to \bar{G}$ that maps $G_1$ and $G_2$ canonically onto $G_1/ P_1$ and $G_2/ P_2$, respectively. Let $U$ be the inverse image with respect to $\theta$ of a free subgroup of finite index in $\bar{G}$. Then $U\leq_f G$ and $U\cap G_i\leq P_i\leq N_i$. Therefore, $G_i\leq_t G$ for $i=1, 2$.
\end{proof}

Corollary 3.2 and Proposition 3.3 provide us with a way to prove that a free product with amalgam is in the class $\mathcal{A}_n$. The only difficulty is that it is not easy to recognize when the conditions stipulated in Proposition 3.3 might be satisfied. Nevertheless, we will discern two important cases where these conditions are fulfilled. The first involves a normal amalgam and is treated in the following theorem, inspired by [{\bf 3}, Proposition 6.1].

\begin{theorem} Assume $G_1$ and $G_2$ are groups in $\mathcal{A}_n$  with a shared finitely generated normal subgroup $N$ in $\mathcal{A}_{n-1}$. If  $G_1/N$ and $G_2/N$ both belong to $\mathcal{A}_2$, then $G_1\ast_N G_2$ is in $\mathcal{A}_n$.
 \end{theorem}
 
 \begin{proof} 
 
 Our plan is to show that, for each pair $\{N_1, N_2\}$ with $N_i\unlhd_f G_i$,  there exists a pair $\{P_1,P_2\}$ such that $P_i\unlhd_f G_i$, $P_i\leq N_i$ and $P_1\cap N=P_2\cap N$. It will then follow by Proposition 3.3 that $G_1$ and $G_2$ are both topologically embedded in  $G_1\ast_N G_2$, yielding the conclusion at once by virtue of Corollary 3.2. To determine the groups $P_i$, we let $U$ be a normal subgroup of $G$ contained in $N\cap N_1\cap N_2$ such that $[N:U]<\infty$-- that such a subgroup exists is a consequence of the fact that $N$ is finitely generated.  Since $N\leq_t G_i$, we can find $M_i\unlhd_f G_i$ such that $M_i\leq N_i$ and $M_i\cap N\leq U$. Now we let $P_i=UM_i$. Then $P_i\cap N=U(M_i\cap N)=U$; moreover, $P_i\unlhd_f G_i$ and $P_i\leq N_i$. Thus we have constructed the desired pair $\{P_1,P_2\}$.
 \end{proof} 

The second situation where the hypotheses of Propositon 3.3 are satisfied is when the amalgamated subgroup is cyclic and the groups are quasipotent, a property defined as follows.

\begin{definition}{\rm  A group $G$ is {\it quasipotent} if, for each $g\in G$, there exists $k\in \mathbb Z^+$ and a sequence $\{N_n\}$ of normal finite index subgroups indexed by $\mathbb Z^+$ such that $\langle g\rangle \cap N_n=\langle g^{nk}\rangle$ for each $n\in \mathbb Z^+$. If, for each $g\in G$, such a sequence $\{N_n\}$ can be chosen so that $N_n$ is a characteristic subgroup of $G$ for all $n\geq 1$, then $G$ is {\it characteristically quasipotent}.}
\end{definition}

Finite groups are, trivially, examples of characteristically quasipotent groups. In \cite{tang} and \cite{segal}, respectively,  it is proven that free groups and polycyclic groups are also in this class. Moreover,  from  [{\bf 2}, Theorem 5.1] one can glean the following proposition, which may be used to construct more examples of groups that are quasipotent and characteristically quasipotent.

\begin{proposition}{\rm (Burillo, Martino)} Let $N\rightarrowtail G\twoheadrightarrow Q$ be a group extension 
with $Q$ a quasipotent group in $\mathcal{A}_2$ and $N$ a finitely generated characteristically quasipotent group. Then $G$ is quasipotent. 

If, in addition, $Q$ is finitely generated and $N$ is a characteristic subgroup of $G$, then $G$ is characteristically quasipotent. 
\end{proposition}

Proposition 3.5 yields, for example, that finitely generated virtually free groups and virtually polycyclic groups are all characteristically quasipotent. Another consequence of this proposition is that (finitely generated free)-by-polycyclic groups are quasipotent.  Additional examples of groups that are quasipotent are provided by [{\bf 2}, Theorems 3.6, 3.7, 3.8].

For our purposes the following property of quasipotent groups will be important; it  can be deduced at once from the definition.

\begin{lemma} If $G$ is a quasipotent group, then every cyclic subgroup of $G$ is topologically embedded in $G$. 
\end{lemma}

Now we examine free products of quasipotent groups with cyclic amalgam, showing that they satisfy the hypotheses of Proposition 3.3.

\begin{lemma} {\rm (L. Ribes and P. Zalesskii \cite{rz})} If $G_1$ and $G_2$ are quasipotent groups with a common cyclic subgroup $A$, then $G_1$ and $G_2$ are both topologically embedded in $G_1\ast_A G_2$.
\end{lemma}

\begin{proof} We need to show that, for each pair $\{N_1,N_2\}$ with $N_i\unlhd_f G_i$, there exists a pair $\{P_1,P_2\}$ such that $P_i\unlhd_f G_i$, $P_i\leq N_i$ and $P_1\cap A=P_2\cap A$. To determine the groups $P_i$, we begin by letting $a$ be a generator of $A$. Since $G_i$ is quasipotent, there exists a sequence $\{_iM_n\}_{n\in \mathbb Z^+}$ of  normal finite-index subgroups of $G_i$ such that $_iM_n\cap A=\langle a^{nk_i}\rangle$ for some $k_i\in \mathbb Z^+$. Moreover, by choosing subsequences of $\{_iM_n\}$, we can make $k_i$ larger than any fixed value. Hence we can select the sequence $\{_iM_n\}$ so that $$_iM_n\cap A=\langle a^{nk_i}\rangle \leq A\cap N_1\cap N_2. $$ Now, if we let $P_1=\ _1M_{k_2}\cap N_1$ and $P_2=\ _2M_{k_1}\cap N_2$, then the pair $\{P_1,P_2\}$ has the desired properties.
\end{proof}

In conjunction with Corollary 3.2, the preceding two lemmas immediately yield the following theorem.

\begin{theorem} Assume $G_1$ and $G_2$ are quasipotent groups with a common cyclic subgroup $A$. If  $G_1$ and $G_2$ are both in $\mathcal{A}_n$, then $G_1\ast_A G_2$ also belongs to $\mathcal{A}_n$.
\end{theorem}

One important special case of the above theorem is provided below.

\begin{corollary} Assume $G_1$ and $G_2$ are groups that are each either virtually free of finite rank or virtually polycyclic. If $G_1$ and $G_2$ have a shared cyclic subgroup $A$, then $G_1\ast_A G_2$ is in $\mathcal{A}_n$ for all $n\geq 0$.
\end{corollary}

The remainder of the section is devoted to HNN extensions. We will employ the following notation for these constructions: given a discrete group $G$ and an isomorphism $\phi:H\to K$, where both $H$ and $K$ are subgroups of $G$,  the HNN extension of $G$ with respect to $\phi$ is denoted by $G_{\phi}$. 
In other words,
$$G_{\phi}=\langle G, t\ |\ t^{-1}ht=\phi(h)\ \mbox{for all $h\in H$} \rangle .$$

In addition to HNN extensions of discrete groups, we will refer to profinite HNN extensions.
As described in \cite{profinite}, from any profinite group $\Gamma$ and any continous isomorphism $\theta: \Delta\to \Lambda$, where $\Delta$ and $\Lambda$ are both closed subgroups of $\Gamma$, we can form the profinite HNN extension of $\Gamma$ with respect to $\theta$. If $\Gamma$, $\Delta$ and $\Lambda$ are each embedded in the profinite HNN extension, we refer to the latter as a {\it proper}  
 profinite HNN extension.  Any proper profinite HNN extension gives rise to a Mayer-Vietoris sequence that relates the cohomology of the extension to that of the group $\Gamma$ and the subgroup $\Delta$. 

Our interest is in the case when $\Gamma=\hat{G}$, $\Delta=\hat{H}$, $\Lambda=\hat{K}$ and $\theta=\hat{\phi}$, where $G$ is a discrete group with topologically embedded subgroups $H$, $K$ and $\phi:H\to K$ is an isomorphism. In this case,  $\hat{G}_{\phi}$ is the profinite HNN extension of $\hat{G}$ with respect to $\hat{\phi}$. Moreover, if $G$ is topologically embedded in $G_{\phi}$, then this profinite HNN extension is proper and, therefore, gives rise to a Mayer-Vietoris sequence. This sequence and its relationship to the discrete Mayer-Vietoris sequence for $G_{\phi}$ are described in the following theorem.

\begin{theorem} Let $G$ be a group with isomorphic, topologically embedded subgroups $H$ and $K$. Assume $\phi:H\to K$ is an isomorphism and $G\leq_t G_{\phi}$. Then, for each discrete $\hat{G}$-module $A$ and positive integer $i$,  we have a commutative diagram
\begin{equation} \begin{CD}
H_c^{i-1}(\hat{H},A) @>>> H_c^i(\hat{G_{\phi}},A) @>>> H_c^i(\hat{G},A)  @>>> H_c^i(\hat{H},A)\\
@VVV @VVV @VVV @VVV \\
H^{i-1}(H,A) @>>> H^i(G_{\phi},A) @>>> H^i(G,A) @>>> H^i(H,A),
\end{CD} \end{equation}
in which the rows are exact and the vertical maps are induced by the completion maps for $G$, $H$ and $G_{\phi}$.
\end{theorem}

Theorem 3.10 provides us with the following conditions under which an HNN extension is in the class $\mathcal{A}_n$. 

\begin{corollary} Let $\phi:H\to K$ be an isomorphism, where $H$ and $K$ are topologically embedded subgroups of a group $G$. Assume, further, that $G\leq_t G_{\phi}$.  If $G$ is in $\mathcal{A}_n$ and $H$ is in $\mathcal{A}_{n-1}$, then $G_{\phi}$ belongs to $\mathcal{A}_n$.
\end{corollary}

\begin{proof} Let $A$ be a finite discrete $\hat{G}$-module, and consider diagram (3.2) for $0\leq i\leq n$. Then the first and third vertical maps are bijective, and the fourth is injective. Therefore, the second map must be surjective, so that $G_{\phi}$ is in $\mathcal{A}_n$.
\end{proof}

Analogous to Proposition 3.3, we have the following conditions that guarantee that the base group is topologically embedded in an HNN extension.

\begin{proposition} Let $\phi:H\to K$ be an isomorphism, where $H$ and $K$ are subgroups of a group $G$. Assume that, for every $N\unlhd_f G$, there exists $P\unlhd_f G$ such that $P\leq N$ and $P\cap K=\phi(P\cap H)$. Then $G\leq_t G_{\phi}$.
\end{proposition}

\begin{proof} Let $N\unlhd_f G$. Then there is a subgroup $P$ satisfying the three conditions stated in the hypothesis. Since 
$P\cap K=\phi(P\cap H)$, $\phi$ induces an isomorphism $\bar{\phi}: PH/P\to PK/P$.  Hence we can form the HNN extension $\bar{G}$ of $G/P$ with respect to $\bar{\phi}$. Moreover, we have an epimorphism $\theta:G_{\phi}\to \bar{G}$ that maps $G$ canonically onto $G/P$. As an HNN extension of a finite group, $\bar{G}$ is virtually free. Let $U$ be the inverse image under $\theta$ of a free subgroup of $\bar{G}$ with finite index. Then $[G_{\phi}:H]< \infty$ and $U\cap G\leq P\leq N$. Therefore, $G\leq_t G_{\phi}$.
\end{proof}

Corollary 3.11 and Proposition 3.12 allow us to prove the following result.

\begin{theorem} Let $G$ be a group and $H\leq_t G$.  Define the group $\Gamma$ by
$$\Gamma=\langle G,t\ |\  [t,h]=1\ \mbox{for all $h\in H$} \rangle.$$
If $G$ and $H$ belong to $\mathcal{A}_n$  and $\mathcal{A}_{n-1}$, respectively, then $\Gamma$ is in $\mathcal{A}_n$.
\end{theorem}

\begin{proof} Applying Proposition 3.12, $G$ is readily seen to be topologically embedded in $\Gamma$. The conclusion then follows by Corollary 3.11.
\end{proof}

As a consequence of this theorem, we can deduce that any right-angled Artin group is contained in $\mathcal{A}_n$ for $n\geq 0$. These groups are defined as follows. 

\begin{definition}{\rm A {\it right-angled Artin group} is any group with a finite generating set $X$ and a presentation of the form
$$\langle X\ |\ [x,y]=1 \ \mbox{\rm for all $(x, y)\in \Sigma$}\ \rangle$$
for some subset $\Sigma$ of the Cartesian product $X\times X$.}
\end{definition}

Our approach to proving that right-angled Artin groups are in $\mathcal{A}_n$ for all $n\geq 0$ is similar to that employed in the proof of [{\bf 2}, Theorem 3.8]. Before proceeding with the proof, we require the following lemma.

\begin{lemma} If $G$ is a right-angled Artin group with generating set $X$, then, for every $X'\subseteq X$, $\langle X'\rangle \leq_t G.$
\end{lemma}

\begin{proof} The proof is by induction on the cardinality of $X$, the case $|X|=1$ being trivial. Assume $|X|>1$. If $X'=X$, then the conclusion follows at once. Assume $X'\neq X$, and let $x\in X-X'$.  Define $H$ to be the group generated by $X-\{ x \}$ with all of the same relators as $G$ except those involving $x$. Furthermore, let $Y$ be the set of all elements in $X-\{ x\}$ that commute with $x$ in $G$. Then
$$G=\langle H, x\ |\ [x,y]=1\ \mbox{for all $y\in Y$}\rangle.$$
Now let $U\leq_f \langle X'\rangle$. By the inductive hypothesis, $\langle X'\rangle\leq_t H$, which means that $H$ contains a normal subgroup $N$ of finite index such that $N\cap \langle X'\rangle\leq U$. Next consider the group
$$\bar{G}=\langle H/N, \bar{x}\ |\ [\bar{x}, Ny]=1\ \mbox{for all $y\in Y$}\rangle.$$
As an HNN extension of a finite group, $\bar{G}$ is virtually free. Moreover, there is a map $\theta:G\to \bar{G}$ that maps $H$ canonically onto $H/N$ and $x$ to $\bar{x}$. Let $V$ be the inverse image under $\theta$ of a free subgroup of finite index in $\bar{G}$. Then $V\leq_f G$ and $V\cap H\leq N$. Thus $V\cap \langle X'\rangle\leq U$. It follows, then, that $\langle X'\rangle \leq_t G$.
\end{proof}

\begin{theorem} Every right-angled Artin group is in $\mathcal{A}_n$ for all $n\geq 0$.
\end{theorem}

\begin{proof}
The proof is by induction on the number of generators, the case of one generator being trivial.  Let $G$ be a right-angled Artin group with generating set $X$ with $|X|>1$, and assume that every right-angled Artin group with fewer generators than $G$ lies in $\mathcal{A}_n$ for all $n\geq 0$.  If $G$ has no relators, then it is free and thus in $\mathcal{A}_n$ for all $n\geq 0$. Suppose $G$ has at least one relator, say $[x_0,y_0]$. Let $[x_0,y_0], [x_1,y_0], \cdots , [x_l,y_0]$ be a list of all the relators that involve $y_0$. 
 Now define $H$ to be the group generated by $X-\{y_0\}$ with all of the same relators as $G$ except those involving $y_0$.   In view of the inductive hypothesis, $H$ must belong to $\mathcal{A}_n$ for all $n\geq 0$. Moreover, 
$$G=\langle H, y_0\ |\ [x_0,y_0]=[x_1,y_0]=\cdots =[x_l, y_0]=1\rangle.$$
By Lemma 3.14, the subgroup of $H$ generated by $x_0,\cdots , x_l$ is topologically embedded in $H$. 
In addition, as a right-angled Artin group with fewer generators than $G$, this subgroup must belong to $\mathcal{A}_n$ for all $n\geq 0$. Therefore, appealing to Theorem 3.13, we can conclude that $G$ is in the class $\mathcal{A}_n$ for all $n\geq 0$.
\end{proof}

 \section{Ascending HNN extensions that are highly residually finite}

In this section we examine a special type of HNN extension, known as an ascending HNN extension and defined as follows.

\begin{definition}{\rm Let $\phi: G\to G$ be a group monomorphism. The {\it ascending HNN extension} of $G$ with respect to $\phi$, denoted $G_{\phi}$, is given by the presentation

$$G_{\phi} = \langle G, t \ |\ t^{-1}gt = \phi(g)\ \mbox{for all}\ g\in G\rangle.$$
}

\end{definition}

 Ascending HNN extensions have the following elementary property, which is an immediate consequence of the normal form for elements of HNN extensions.

\begin{lemma} Let $\phi:G\to G_{\phi}$ be a group monomorphism. Then every element of $G_{\phi}$ can be expressed in the form $t^kgt^{-l}$, where $g\in G$ and $k, l\in \mathbb Z^+\cup \{0\}$.
\end{lemma}

An ascending HNN extension of a residually finite group may not be residually finite; see, for instance, the examples described in 
\cite{sapwise} as well as [{\bf 8}, Theorem 1.2]. However, as established in \cite{sapir, rhem, wise}, ascending HNN extensions of some important special types of  
residually finite groups are always residually finite. Among these special types are finitely generated free groups, whose ascending HNN extensions are investigated in \cite{sapir} using methods from algebraic geometry.

\begin{theorem}{\rm (A. Borisov, M. Sapir)} If $G$ is a finitely generated free group and $\phi:G\to G$  a monomorphism, then $G_{\phi}$ is residually finite.
\end{theorem}

Our objective in this section is to strengthen the above result by showing that ascending HNN extensions of finitely generated free groups are actually highly residually finite. In addition, we consider ascending HNN extensions of polycyclic groups, proving that they, too, are highly residually finite.  Finally, we generalize these two results by establishing the highly residual finiteness of certain ascending HNN extensions of groups possessing a normal series whose factor groups are each either free, polycyclic or finite. The techniques we employ in this section eschew any explicit reference to cohomology; instead we rely on the third characterization of highly residually finite groups given in Theorem 2.10.

We begin with the following elementary property of subgroups of finite index in a finitely generated group.

\begin{lemma} Let $G$ be a finitely generated group, $N\unlhd_{f}G$ and $\phi:G\to G$ a homomorphism. Then $N$ contains a subgroup $M\unlhd_f G$ such that $\phi(M)\leq M$.
\end{lemma}

\begin{proof} For each nonnegative integer $i$, let
$$\phi^{-i}(N)=\{ x\in G : \phi^i(N)\in N \},$$
where $\phi^0$ is understood to be the identity map from $G$ to $G$.
Now set $M=\bigcap_{i=0}^{\infty}\phi^{-i}(N)$. It is easy to see that $M\leq N$, $M\unlhd G$ and $\phi(M)\leq M$. All that remains to be shown, then, is that
$[G:M]<\infty$. To establish this, we first observe that, for each $i\geq 0$, $[G:\phi^{-i}(N)]\leq [G : N].$ Since 
$G$ possesses only finitely many subgroups with index $\leq [G:N]$, it follows that there are only finitely many subgroups of the 
form  $\phi^{-i}(N)$ for $i\geq 0$. Therefore, $M$, as the intersection of finitely many subgroups with finite index, has finite index.
\end{proof}

Our analysis of ascending HNN extensions is based on the following property of a group endomorphism.

\begin{definition}{\rm  A group endomorphism $\phi:G\to G$ has property $\mathfrak{P}$ if, for every $g\in G$, there exists $N\unlhd_{f}G$ such that,  for $i=0,1,\cdots$,  $$\phi^i(g)\in N \Leftrightarrow \phi^i(g)=1.$$}
\end{definition}

For injective endomorphisms the above property can be expressed as follows. 

\begin{lemma} A group monomorphism $\phi:G\to G$ has property $\mathfrak{P}$ if and only if, for every nontrivial element $g$ in $G$, there exists $N\lhd_{f}G$ such that $\phi^i(g)\notin N$ for all $i\geq 0$. 
\end{lemma}

The significance of property $\mathfrak{P}$ for the study of ascending HNN extensions is revealed by the following proposition.

\begin{proposition} Let $G$ be a finitely generated group and $\phi:G\to G$ a monomorphism. The group $G_{\phi}$ is residually finite if and only if $\phi$ has property $\mathfrak{P}$. 
\end{proposition}

\begin{proof} ($\Longrightarrow$) Assume $g\in G$ with $g\neq 1$. Let $M\lhd_{f}G_{\phi}$ such that $g\notin M$. Set $N=M\cap G$. If $\phi^i(g)\in N$ for some $i\geq 0$, then $t^{-i}gt^i=\phi^i(g)\in N\leq M$, which implies that $g\in N$, a contradiction. Hence $\phi^i(g)\notin N$ for all $i\geq 0$. Therefore, $\phi$ has the desired property.   
 \vspace{10pt}

($\Longleftarrow$)  Let $x\in G_{\phi}$ with $x\neq 1$. We require a subgroup of finite index in $G_{\phi}$ that misses $x$. By Lemma 4.1,  $x=t^kgt^{-l}$ for some $g\in G$, where $k$ and $l$ are nonnegative integers.  First we dispose of the case when $g=1$, i.e., $x=t^n$ for $n\in \mathbb Z-\{0\}$. Let $\epsilon : G_{\phi}\to \mathbb Z$ be the epimorphism that maps $G$ to $\{0\}$ and $t$ to $1$.  Then, if $A\leq_f \mathbb Z$ such that $n\notin A$, we can take $\epsilon^{-1}(A)$ as the desired subgroup. Next we treat the really serious case, namely, when $g\neq 1$. By Lemmas 4.4 and 4.3, we can find $N\lhd_f G$ such that $\phi(N)\leq N$ and $\phi^{i}(g)\notin N$ for $i\geq 0$. Let $M=\bigcup_{i=0}^{\infty}\phi^{-i}(N)$, where $\phi^{-i}(N)$ is as defined above in the proof of Lemma 4.3. Then $g\notin M$, $M\lhd_f G$ and $\phi(M)\leq M$. In addition, the map $G/M\to G/M$ induced by $\phi$ is an isomorphism. Hence we can form the semidirect product $G/M\rtimes \langle \bar{t}\rangle$ in which $\bar{t}$ has infinite order and $\bar{t}^{-1}(My)\bar{t}=M\phi(y)$ for all $y\in G$. Moreover, we have a map $\theta:G_{\phi}\to G/M\rtimes \langle \bar{t}\rangle$ that maps $G$ canonically onto $G/M$ and $t$ to $\bar{t}$. Now let $H= \theta^{-1}(\langle \bar{t}\rangle)$. Then $H\leq_f G_{\phi}$ and $g\notin H$. Also, $t\in H$, so that $x\notin H$, thus allowing $H$ to serve as the required subgroup.
\end{proof}

In order to detect the above group-theoretic property,  it is desirable to understand when this property is inherited by extensions. In this respect the following lemma will prove useful.

\begin{lemma} Assume $N\rightarrowtail G\twoheadrightarrow Q$ is a group extension such that $N$ is finitely generated and residually finite and $Q$ is highly residually finite. Let $\phi:G\to G$ be a homomorphism such that $\phi(N)\leq N$. If the homomorphisms $Q\to Q$ and $N\to N$ induced by $\phi$ both possess property $\mathfrak{P}$, then $\phi$ must also have property $\mathfrak{P}$ . 
\end{lemma}

\begin{proof}  Let $g\in G$ with $g\neq 1$.  We require a normal subgroup of finite index that misses all the nontrivial elements of the form $\phi^i(g)$ for $i\geq 0$. If $\phi^k(g)=1$ for some $k>0$, then $\phi^i(g)=1$ for all $i\geq k$, which means that the existence of such a subgroup follows immediately from the residual finiteness of $G$. Thus we only need consider the case when $\phi^i(g)\neq 1$ for all $i\geq 0$. First suppose that
 $\phi^i(g)\in N$ for some $i\geq 0$, and let $k$ be the smallest nonnegative integer such that $\phi^k(g)\in N$. Then, because $\phi$ induces a map with property $\mathfrak{P}$ on $N$, there exists $M\lhd_{f}N $ such that $\phi^i(g)\notin M$ for $i\geq k$. Moreover, since $Q$ lies in $\mathcal{A}_2$, we can find $U_1\lhd_{f}G$ such that $U_1\cap N\leq M$. Therefore, $\phi^i(g)\notin U_1$ for all $i\geq k$. If $k=0$, then we are done. For $k>0$ we can invoke the residual finiteness of $G$ to obtain $U_2\lhd_{f}G$ such that $\phi^i(g)\notin U_2$ for $0\leq i<k$. Then, setting $U=U_1\cap U_2$,  we have $U\lhd_{f}G$ and $\phi^i(g)\notin U$ for all $i\geq 0$. Finally, in the case where $\phi^i(g)\notin N$ for all $i\geq 0$, we can apply the hypothesis that the map $Q\to Q$ induced by $\phi$ has property $\mathfrak{P}$ in order to produce a subgroup of  $G$ with the desired properties.
\end{proof}

Lemma 4.6 can be used to prove a result analogous to Proposition 4.5 that provides a way to determine when an ascending HNN extension of a highly residually finite group is itself highly residually finite.

\begin{proposition} Let $G$ be a finitely generated highly residually finite group. If $\phi:G\to G$ is a monomorphism with property $\mathfrak{P}$, then $G_{\phi}$ is highly residually finite.
\end{proposition}

\begin{proof} Assume $F\rightarrowtail E\stackrel{\epsilon}{\twoheadrightarrow} G_{\phi}$ is a group extension with $F$ finite. Set $H=\epsilon^{-1}(G)$, and let $u\in E$ such that $\epsilon(u)=t$. Conjugation by $u$ induces a monomorphism $\psi: H\to H$ such that $\psi (F)=F$. Moreover, we have $E\cong H_{\psi}$. Also, since $\psi$ induces $\phi:G\to G$, $\psi$ has property $\mathfrak{P}$ by Lemma 4.6, so that $E$ is residually finite. Therefore, $G_{\phi}$ is highly residually finite.
\end{proof}

The above proposition, combined with Theorem 4.2 and Proposition 4.5, yields the following result.

\begin{theorem} If $G$ is a free group of finite rank and $\phi:G\to G$ a monomorphism, then $G_{\phi}$ is highly residually finite.
\end{theorem}

Now we proceed to identify some classes of groups for which every endomorphism has property $\mathfrak{P}$. The first is the class of finite groups, for which every endomorphism satisfies the property trivially.
 
\begin{lemma} If $G$ is a finite group, then every endomorphism of $G$ has property $\mathfrak{P}$.
\end{lemma}

Next we consider endomorphisms of finitely generated abelian groups.

\begin{lemma} If $A$ is a finitely generated abelian group, then any homomorphism $\phi:A\to A$ has property $\mathfrak{P}$. 
\end{lemma}

\begin{proof} First we consider the case when $\phi$ is a monomorphism. We may decompose $A$ as $A= F\oplus \bar{A}$ such that $\bar{A}$ is the direct sum of finitely many copies of $\mathbb Z$ and $F$ is finite. Let $\bar{\phi}:\bar{A}\to \bar{A}$ be the monomorphism induced by $\phi$. Assume  $a\in A$ with $a\neq 0$. We need to find a subgroup of finite index in $A$ that does not contain $\phi^i(a)$ for all $i\geq 0$.  Let $\bar{a}$ be the $\bar{A}$-portion of $a$. If $\bar{a}=0$, then $\bar{A}$ is a subgroup with the properties we require.
Suppose $\bar{a}\neq 0$. Let $p$ be a prime such that $p$ does not divide all of the components of $\bar{a}$ and also does not divide $\mbox{det}\ \bar{\phi}$. Then, for each integer $i\geq 0$,  the prime $p$ does not divide all of the components of $\bar{\phi}^i(\bar{a})$. Hence $\phi^i(a)\notin pA$
for all $i\geq 0$, so that $pA$ can serve as the desired subgroup.

Now consider the case where $\phi$ is not injective. Let $C=\bigcup_{i=1}^{\infty}\mbox{Ker}\ \phi^i$.
Then $C\leq A$ and $\phi$ induces a monomorphism $\phi^*: A/C\to A/C$. Assume $a\in A$ and $a\neq 0$.  We need to establish the existence of a subgroup $D<_f A$ such that $\phi^i(a)\in D\Leftrightarrow \phi^i(a)=0$.  If $a\notin C$, we can obtain such a subgroup by applying the result for monomorphisms to $\phi^*$. 
Now suppose $a\in C$. Let $k$ be the smallest positive integer such that $\phi^k(a)=0$. Since $A$ is residually finite, there exists $D<_fA$ such that $\phi^i(a)\notin D$ for $0\leq i<k$, giving us the desired subgroup. 
\end{proof}

With the aid of Lemma 4.6, we can easily extend the above result to polycyclic groups.

\begin{lemma}If $G$ is a polycyclic group, then any endomorphism $\phi:G\to G$ has property $\mathfrak{P}$.
\end{lemma}

\begin{proof} We proceed by induction on the length of the derived series of $G$, the base case having been established in Lemma 4.10. Assume $G'\neq 1$ and consider the extension $G'\rightarrowtail G\twoheadrightarrow G_{ab}$. This extension  satisfies all of the hypotheses of Lemma 4.6 with respect to any endomorphism $\phi:G\to G$; thus the conclusion follows.
\end{proof}

As an immediate consequence of Lemma 4.11, we obtain that ascending HNN extensions of polycyclic groups are highly residually finite.

\begin{theorem} If $G$ is a polycyclic group and $\phi:G\to G$ a monomorphism, then $G_\phi$ is highly residually finite.
\end{theorem}

Theorem 4.12 can also be deduced readily from T. Hsu and D. Wise's \cite{wise} theorem that ascending HNN extensions of virtually polycyclic groups are residually finite.  Nevertheless, we have adopted our approach because, in addition to providing an alternative, shorter proof for  Hsu and Wise's result, it allows us to prove a theorem that generalizes both Theorem  4.12 and Theorem 4.8. Before we can accomplish this, however, we need to prove that endomorphisms of finitely generated free groups satisfy property $\mathfrak{P}$ as well.

\begin{lemma} If $G$ is a finitely generated free group and $\phi:G\to G$ an endomorphism, then $\phi$ has property $\mathfrak{P}$.
\end{lemma} 

\begin{proof} For each $i\geq 0$,  $\phi^i(G)$ is a finitely generated free group; moreover, the rank of $\phi^{i+1}(G)$ does not exceed the rank of $\phi^i(G)$. Hence there exists $k\geq 0$ such that the rank of $\phi^k(G)$ is equal to the rank of $\phi^{k+1}(G)$.  It follows, then, from the fact that finitely generated free groups are hopfian that $\phi$ induces an injection from $\phi^k(G)\to \phi^{k+1}(G)$. Hence $\mbox{Ker}\ \phi^{k+1}=\mbox{ Ker}\ \phi^k$. Now let $K=\mbox{Ker}\ \phi^k$. Then $\phi (K)\leq K$, and $\phi$ induces a monomorphism $\bar{\phi}: G/K \to G/K$. Moreover, $G/K\cong \phi^k(G)$, so that $G/K$ is free of finite rank. Thus, by Theorem 4.2 and Proposition 4.5, $\bar{\phi}$ has property $\mathfrak{P}$. 

Now we are ready to show that $\phi$, too, has property $\mathfrak{P}$.  Let $g\in G$ with $g\neq 1$. We desire a normal subgroup of finite index that misses all the nontrivial elements of the form $\phi^i(g)$ for $i\geq 0$.  If $g\notin K$, then such a subgroup may be obtained by invoking the fact that $\bar{\phi}:G/K\to G/K$ has property $\mathfrak{P}$.
Now consider the case where $g\in K$. Let $l$ be the smallest positive integer such that $\phi^l(g)=1$. Since $G$ is residually finite, there exists $N\lhd_f G$ such that $\phi^i(g)\notin N$ for $0\leq i<l$. The subgroup $N$, then,  enjoys the properties we seek.
\end{proof}

 Armed with the above lemma,  we are prepared to prove our most general result concerning ascending HNN extensions.

\begin{theorem} Let
$$1=G_0\unlhd G_1\unlhd G_2\unlhd \cdots G_{r-1}\unlhd G_r=G$$
be a normal series in a group $G$ such that each factor group $G_i/G_{i-1}$ is either finite, free of finite rank or polycyclic. If $\phi:G\to G$ is a monomorphism such that $\phi(G_i)\leq G_i$ for $1\leq i\leq r$, then $G_{\phi}$ is highly residually finite.
\end{theorem}

\begin{proof} This follows immediately from the lemma below.
\end{proof}

\begin{lemma} Let
$$1=G_0\unlhd G_1\unlhd G_2\unlhd \cdots G_{r-1}\unlhd G_r=G$$
be a normal series in a group $G$ such that each factor group $G_i/G_{i-1}$ is either finite, free of finite rank or polycyclic. If $\phi:G\to G$ is a homomorphism such that $\phi(G_i)\leq G_i$ for $1\leq i\leq r$, then $\phi$ has property $\mathfrak{P}$. 
\end{lemma}

\begin{proof} We induct on $r$, the case $r=1$ having already been established. Assume $r>1$. Consider the extension $G_{r-1}\rightarrowtail G\twoheadrightarrow G/G_{r-1}$. By Lemmas 4.9, 4.11 and 4.13, the endomorphism of  $G/G_{r-1}$ induced by $\phi$ has property $\mathfrak{P}$; moreover, the map from $G_{r-1}$ to $G_{r-1}$ induced by $\phi$ also has property $\mathfrak{P}$ by the inductive hypothesis. Therefore, since $G/G_{r-1}$ is highly residually finite, the conclusion holds by 
Lemma 4.6.
\end{proof}

Two special cases of Theorem 4.14 are worth emphasizing.

\begin{corollary}
If $G$ is a finitely generated virtually free group and $\phi:G\to G$ a monomorphism, then $G_{\phi}$ is highly residually finite.
\end{corollary}

\begin{corollary} If $G$ is a virtually polycyclic group and $\phi:G\to G$ a monomorphism, then $G_{\phi}$ is highly residually finite.
\end{corollary}

In conclusion, we remark that it remains unknown whether the ascending HNN extensions shown to be highly residually finite in this section are also in $\mathcal{A}_n$ for $n>2$.  Unfortunately, the Mayer-Vietoris sequence, employed in the previous section to identify groups in $\mathcal{A}_n$ for $n>2$, sheds no light on this question, because, in an ascending HNN extension, the base group may not be topologically embedded.

\vspace{20pt}

\begin{acknowledgement}{\rm  The author would like to thank the Department of Analysis and Scientific Computing of the Technical University of Vienna for its hospitality while this work was in progress.  In particular, he is deeply grateful to Wolfgang Herfort for being such a generous, accommodating host and for his perspicacious suggestions regarding this article.}
\end{acknowledgement}

\end{document}